\documentclass[10pt,a4paper]{article}
\usepackage{amsmath,amssymb,amsfonts,amsthm}
\usepackage{float,graphicx,color}
\usepackage[all]{xy}

\newcommand{\rmH}{\mathrm{H}}
\newcommand{\rmL}{\mathrm{L}}

\newcommand{\bbC}{\mathbb{C}}

\newcommand{\bbN}{\mathbb{N}}

\newcommand{\bbR}{\mathbb{R}}

\newcommand{\bbZ}{\mathbb{Z}}

\newcommand{\rmins}[1]{\ \mathrm{#1}\ }

\newcommand{\ts}{\textstyle}

\newcommand{\myand}{\quad \mathrm{and} \quad }

\newcommand{\cleq}{\preccurlyeq}
\newcommand{\myP}{P}
\newcommand{\mys}{\sharp}
\newcommand{\myss}{${}^\sharp$ }

\DeclareMathOperator{\rmd}{d}

\DeclareMathOperator{\id}{id}

\DeclareMathOperator{\grad}{grad}
\DeclareMathOperator{\curl}{curl}
\DeclareMathOperator{\Div}{div}
\renewcommand{\div}{\Div}

\newtheorem{proposition}{Proposition}
\newtheorem{lemma}[proposition]{Lemma}

\theoremstyle{definition}\newtheorem{definition}{Definition}
\theoremstyle{remark}\newtheorem{remark}{Remark}

\numberwithin{equation}{section}

\begin{document}
\title{On variational eigenvalue approximation\\
of semidefinite operators}
\author{Snorre H. Christiansen, Ragnar Winther}

\date{}

\maketitle

\begin{abstract}
Eigenvalue problems for semidefinite operators with infinite dimensional kernels appear for instance in electromagnetics. Variational discretizations with edge elements have long been analyzed in terms of a discrete compactness property. As an alternative, we show here how the abstract theory can be developed in terms of a geometric property called the vanishing gap condition. This condition is shown to be equivalent to eigenvalue convergence and intermediate between two different discrete variants of Friedrichs estimates. Next we turn to a more practical means of checking these properties. We introduce a notion of compatible operator and show how the previous conditions are equivalent to the existence of such operators with various convergence properties. In particular the vanishing gap condition is shown to be equivalent to the existence of compatible operators satisfying an Aubin-Nitsche estimate. Finally we give examples demonstrating that the implications not shown to be equivalences, indeed are not.  
\end{abstract}

\begin{quote}
MSC classes : 65J10, 65N25, 65N30.
\end{quote}

\section{Introduction}
The basic eigenvalue problem in electromagnetics reads:
\begin{equation}
\curl \mu^{-1} \curl u = \omega^2 \epsilon u,
\end{equation}
for some vector field $u$ satisfying boundary conditions in a  bounded domain in Euclidean space. The matrix coefficients $\mu$ and $\epsilon$ characterize electromagnetic properties of the medium. The eigenvalue $\omega^2$ is expressed here in terms of the angular frequency $\omega$.
The operator on the left is semidefinite with an infinite dimensional kernel. The variational approximation of such eigenvalue problems reads: find $u$ and $\omega$ such that for all $u'$ there holds:
\begin{equation}\label{eq:maxwell}
\int \mu^{-1}\curl u \cdot \curl u' = \omega^2 \int\epsilon u \cdot u'.
\end{equation}
In Galerkin discretizations, $u$ is in some finite dimensional space $X_n$ and the above equation should hold for all $u'$ in $X_n$. One has an integer parameter $n$ and looks at convergence properties as $n$ increases.

The above eigenvalue problem is of the following general form: find $u\in X$ and $\lambda \in \bbR$ such that for all $u'\in X$:
\begin{equation}
a(u, u') = \lambda \langle u, u' \rangle,
\end{equation}
where on the right we have the scalar product $\langle \cdot, \cdot \rangle$ of a Hilbert space $O$ (typically an $\rmL^2$ space) and on the left a symmetric semipositive bilinear form $a$ on a Hilbert space $X$. In our setting, $X$ will be dense and continuously embedded in $O$, but not necessarily compactly so, so that the standard theory \cite{BabOsb91} does not directly apply. The important hypothesis is instead that the $O$-orthogonal complement of the kernel of $a$ in $X$ is compactly embedded in $O$. In the context of electromagnetics, such compactness results are included in the discussion of \cite{Cos90}.

The most successful Galerkin spaces for (\ref{eq:maxwell}) are those of \cite{Ned80} and the convergence of discrete eigenvalues has been shown to follow essentially from a \emph{discrete compactness} (DC) property of these spaces, as defined in \cite{Kik89}. This notion has been related to that of collectively compact operators \cite{MonDem01}. Under some assumptions, eigenvalue convergence is actually equivalent to discrete compactness \cite{CaoFerRaf00}.  Necessary and sufficient conditions for eigenvalue convergence are also obtained in \cite{BofBreGas97} via mixed formulations. Natural looking finite element spaces, which satisfy standard inf-sup conditions, equivalent to a \emph{discrete Friedrichs}  estimate (DF), but nevertheless yielding spurious eigenvalues, have been exhibited \cite{BofBreGas00}. For reviews see \cite{Hip02,Mon03,Bof10}.

Eigenpair convergence has long been expressed in terms of gaps between discrete and continuous eigenspaces. In \cite{Chr04MC} the analysis of some surface integral operators was based upon another type of gap, which, transposed to the above problem, concerns the distance from discrete divergence free vector fields to truly divergence free ones. In other words, one considers the gap, in its unsymmetrized form, between the discrete and continuous spaces spanned by the eigenvectors attached to non-zero eigenvalues, from the former to the latter. The general framework was further developed in \cite{BufChr03}, to the effect that discrete inf-sup conditions for non-coercive operators naturally followed from a \emph{vanishing gap} (VG) condition. A variant can be found in \cite{Chr02p}. In \cite{Buf05} this theory was applied to the source problem for (\ref{eq:maxwell}) in anisotropic media. Contrary to DC, VG has a clear visual interpretation, expressing a geometric property of abstract discrete Hodge decompositions.

That VG implies DC is immediate. The converse was noted in \cite{BufPer06}. Thus VG is, via the above cited results, equivalent to eigenvalue convergence.  Moreover DC implies DF, though \cite{BofBreGas00} shows that the converse is not true. However we shall show that VG follows from an optimal version of DF, referred to as ODF henceforth. While unimportant for eigenvalue convergence, the ODF, in the form of a negative norm estimate, was used in \cite{Chr05} (under the name uniform norm equivalence) to prove a discrete div-curl lemma. It was also remarked that a local version of the ODF is implied by the discrete div-curl lemma. A common underlying assumption is that we have approximating discrete kernels (ADK). In section \S \ref{sec:frg} we detail the relationships between ADK, DF, DC (several variants), VG and ODF. We also show how to deduce eigenvalue convergence from VG. While VG uses the norm of $X$ to compute the gap, we relate  ODF to a gap property in $O$. The results on ODF are new. Most of the other results should be considered known in principle, but the ordering and brevity of the arguments might be of interest. A summary is provided in diagram (\ref{eq:summary}).

The main tool used to prove DC and error estimates for finite element spaces is the construction of interpolation operators. The role of commuting diagrams they satisfy has been highlighted \cite{Bof00}. However the interpolators deduced from the canonical degrees of freedom of edge elements are not even defined on $\rmH^1$ for domains in $\bbR^3$. This has lead to great many technical hurdles, involving delicate regularity estimates. Commuting projectors that are \emph{tame}, in the sense of being uniformly bounded $\rmL^2 \to \rmL^2$, were proposed in \cite{Sch08}. Another construction combining standard interpolation with a smoothing operator, obtained by cut-off and convolution on reference macro-elements, was introduced in \cite{Chr07NM}. As pointed out in \cite{ArnFalWin06}, for flat domains, quasi-uniform grids and natural boundary conditions, one can simplify the construction to use only smoothing by convolution on the entire physical domain. In \cite{ChrWin08} these restrictions were overcome by the introduction of a space dependent smoother. Such operators yield eigenvalue convergence quite easily, as well as ODF. It should also be remarked that eigenvalue convergence for the discretization of the Hodge-Laplacian by Whitney forms had been obtained much earlier in \cite{DodPat76}. For a review of the connection between finite elements and Hodge theory, see \cite{ArnFalWin06,ArnFalWin10}.

Reciprocally one would like to know if DF, VG or ODF imply the existence of commuting projections with enough boundedness, since this will indicate possible strategies for convergence proofs in more concrete settings. For instance, in \S 4 of \cite{BofBreGas97}, it is shown that eigenvalue convergence for mixed formulations is equivalent to the existence of a Fortin operator, converging in a certain operator norm. In \S 3.3 of \cite{ArnFalWin10} it is shown, in the context of Hilbert complexes, that DF is equivalent to the existence of commuting projections that are uniformly bounded in energy norm. We introduce here a notion of \emph{compatible operator} (CO) which contains both Fortin operators and commuting projections.  With this notion at hand, DF is equivalent to the existence of energy bounded COs, VG is equivalent to the existence of COs satisfying an Aubin-Nitsche estimate, and ODF is equivalent to the existence of tame COs. The equivalences of DF, VG and ODF to various properties of compatible operators are detailed in \S \ref{sec:projections}.

In \S \ref{sec:add} we give some complementary results. \S \ref{sec:hodge} shows how tame compatible operators appear in the context of differential complexes. \S \ref{sec:eqfor} studies what happens when the bilinear forms are replaced by equivalent ones, extending a result proved in \cite{CaoFerRaf00}. \S \ref{sec:strict} considers the optimality of proved implications. We show how to construct approximating subspaces satisfying ADK but not  DF. We also show how one can construct spaces satisfying DF but not DC, and spaces satisfying DC but not ODF. Thus all the major proved implications that were not proved to be equivalences, are proved not to be. These results are new in this abstract form, but as already mentioned, \cite{BofBreGas00} gives a concrete counter-example of spaces satisfying DC but not DF.

\section{Setting\label{sec:setting}}

\subsection{Continuous spaces and operators}

Let $O$ be an infinite dimensional separable real Hilbert space with scalar product $\langle \cdot,
\cdot \rangle$. Let $X$ be a dense subspace which is itself a Hilbert
space with continuous inclusion $X \to O$. We suppose that we have a continuous symmetric
bilinear form $a$ on $X$ satisfying:
\begin{equation}
\forall u \in X \quad a(u,u) \geq 0.
\end{equation}
Moreover we suppose that the bilinear form $\langle \cdot, \cdot \rangle + a$ is
coercive on $X$ so that we may take it to be its scalar product. The
corresponding norm on $X$ is denoted $\| \cdot \|$, whereas the
natural one on $O$ is denoted $| \cdot | $. Thus:
\begin{align}
|u|^2 &= \langle u, u \rangle,\\
\label{eq:scalpx}
\| u\|^2 &= \langle u, u \rangle + a(u,u).
\end{align} 

Define:
\begin{align}
W&= \{ u \in X \ : \ \forall u' \in X \quad a(u,u')= 0 \},\\
V&= \{u \in X \ : \ \forall w \in W \quad \langle u , w \rangle = 0 \}, 
\end{align}
so that we have a direct sum decomposition into closed subspaces:
\begin{equation}\label{eq:splitx}
X= V \oplus W,
\end{equation}
which are orthogonal with respect to both $\langle \cdot, \cdot \rangle $ and $a$. 

Since $a$ is semi-definite we have:
\begin{equation}
W= \{ u \in X \ : \ a(u,u)= 0 \}.
\end{equation}
We do \emph{not} require $W$ to be finite-dimensional. Remark that $W$ is closed also in $O$. Let $\overline V$ be the
closure of $V$ in $O$ and remark:
\begin{equation}
\overline V = \{u \in O \ : \ \forall w \in W \quad \langle u , w \rangle = 0 \},
\end{equation}
so that we have an orthogonal splitting of $O$:
\begin{equation}
O = \overline V \oplus W.
\end{equation}
We also have:
\begin{equation}\label{eq:intersect}
V = \overline V \cap X.
\end{equation}

We let $P$ be the projector in $O$ with range $\overline V$ and kernel
$W$. It is orthogonal in $O$. It maps $V$ to $V$ so it may also be regarded as a continuous
projector in $X$. Since the splitting (\ref{eq:splitx}) is orthogonal with respect to the scalar product on $X$ defined by (\ref{eq:scalpx}), $P$ is an orthogonal projector also in $X$.

\emph{We suppose that the injection $V \to O$ is compact} (when $V$ inherits
the norm topology of $X$). It implies the Friedrich inequality, namely that there is $C > 0$ such that:
\begin{equation}\label{eq:frie}
\forall v \in V \quad |v|^2 \leq C a(v,v).
\end{equation} 
In particular, on $V$, $a$ is a scalar product whose associated norm is equivalent to the one inherited from $X$, defined by (\ref{eq:scalpx}). In $X$, $P$ can be characterized by the property that for $u\in X$, $Pu \in V$ solves:
\begin{equation}
a(Pu,v) = a(u,v),
\end{equation}
for all $v \in V$. This holds then for all $v \in X$. We will also frequently use the identity, for $u,v \in X$:
\begin{equation}\label{eq:app}
a(Pu, Pv)=a(u,v).
\end{equation}

\begin{remark} \label{rem:exoneandtwo}  Let $S$ be a bounded contractible Lipschitz domain in $\bbR^3$ with outward pointing normal $\nu$ on $\partial S$. Equation (\ref{eq:maxwell}) leads to the following functional frameworks, depending on which boundary conditions are used. For simplicity, we take $\epsilon$ and $\mu$ to be scalar Lipschitz functions on $S$ that are bounded below, above zero. 
\begin{itemize} 
\item 
Define:
\begin{align}
O & = \rmL^2(S) \otimes \bbR^3, \\
X & = \{ u \in O \ : \ \curl u \in O \},\\
a(u,u') & = \ts \int \mu^{-1} \curl u \cdot \curl u',\\
\langle u , u' \rangle & = \ts \int \epsilon u \cdot u'.
\end{align}
Then we have:
\begin{align}
W& = \grad \rmH^1 (S),\\
\overline V &= \{ u \in O \  : \ \div \epsilon u = 0 \rmins{and} u \cdot \nu = 0 \}.
\end{align}
 
\item As a variant, replace in the above definitions:
\begin{equation}  
X  = \{ u \in O \ : \ \curl u \in O \rmins{and}  u \times \nu = 0 \}.
\end{equation}
Then we get:
\begin{align}
W&= \grad \rmH^1_0(S),\\
\overline V &= \{ u \in O \  : \ \div \epsilon u = 0 \}.
\end{align}
\end{itemize}
In both examples, compactness of $V \to O$ is guaranteed by results in \cite{Cos90}. If the topology of $S$ is more complicated, $W$ can contain, in addition to gradients, a non-trivial albeit finite dimensional, space of harmonic vectorfields. The characterization of $\overline V$ should then be modified accordingly.

For the case of more general coefficients $\epsilon$ and $\mu$, see \S \ref{sec:eqfor}.
\end{remark}

The dual of $X$ with $O$ as pivot space is denoted $X'$. The duality pairing on $X' \times X$ is thus denoted $\langle \cdot, \cdot \rangle$. We let $V'$ denote the subspace of $X'$ consisting of elements $u$ such that $\langle u, v \rangle = 0$ for all $v \in W$. 

Let $K: X' \to X$ be the continuous operator defined as follows. For all $u\in X'$, $Ku
\in V$ satisfies:
\begin{equation}\label{eq:kudef}
\forall v\in V \quad a(Ku, v) = \langle u, v \rangle.
\end{equation}
If $u\in V'$, equation  (\ref{eq:kudef}) holds for all $v\in X$. One sees that $K$ is identically $0$ on $W$ and induces an isomorphism $V' \to V$.

Since $V \to O$ is compact, $K : X' \to O$ is compact. As an operator $O \to O$, $K$ is selfadjoint. We deduce that $K$ is compact $ O \to X$ and a fortiori $O \to O$. We are interested in the following eigenvalue problem. Find $u \in X$
and $\lambda \in \bbR$ such that:
\begin{equation}
\forall u'\in X \quad a(u,u')= \lambda \langle u, u' \rangle.
\end{equation}
When $W \neq 0$ it is the eigenspace associated with the eigenvalue
$0$. A positive $\lambda$ is an eigenvalue of $a$ iff $1/\lambda$ is
a positive eigenvalue of $K$; the eigenspaces are the same. The non-zero eigenvalues of $a$ form an increasing positive sequence
diverging to infinity, each eigenspace being a finite
dimensional subspace of $V$. The direct sum of these eigenspaces is dense in
$V$.

\subsection{Discretization}

Given the above setting, consider a sequence $(X_n)$ of finite dimensional subspaces of $X$. We will later formulate conditions to the effect that elements $u$ of $X$ can be approximated by sequences $(u_n)$ such that $u_n \in X_n$. In the mean-time we state some algebraic definitions.

We consider the following discrete eigenproblems.
Find $u \in X_n$
and $\lambda \in \bbR$ such that:
\begin{equation}\label{eq:disceig}
\forall u'\in X_n \quad a(u,u')= \lambda \langle u, u' \rangle.
\end{equation}
We decompose:
\begin{align}
W_n&= \{ u \in X_n \ : \ \forall u' \in X_n \quad a(u,u')= 0 \} = X_n \cap W,\\
V_n&= \{u \in X_n \ : \ \forall w \in W_n \quad \langle u , w \rangle = 0 \}, 
\end{align}
so that:
\begin{equation}\label{eq:split}
X_n = V_n \oplus W_n.
\end{equation} 
Thus $W_n\subseteq W$ but it's crucial, for the points we want to make, that $V_n$ need not be a subspace of $V$.

Notice that $W_n$ is the $0$ eigenspace and that for non-zero $v \in V_n$ we have:
\begin{equation}
a(v,v) > 0.
\end{equation}
Define $K_n$ as follows. For all $u \in X'$, $K_n u \in V_n$ satisfies:
\begin{equation}\label{eq:defkn}
\forall v \in V_n \quad a(K_n u, v) = \langle u, v \rangle.
\end{equation}
Define also $P_n : X \to V_n$ by, for all $u \in X$, $P_nu \in V_n$
satisfies:
\begin{equation}\label{eq:defpn} 
\forall v \in V_n \quad a(P_n u, v) = a(u, v),
\end{equation}
which then holds for all $v \in X_n$.

For $u \in V'$ we have, for all $v \in V_n$:
\begin{equation}
a(P_nKu,v)= a(Ku,v)= \langle u, v \rangle = a(K_n u,v).
\end{equation}
So that:
\begin{equation}\label{eq:knpnk}
K_n u = P_n K u.
\end{equation}
If, on the other hand $u \in W$, $K u =0$ so $P_nKu=0$, but $K_n u = 0$ iff $u$ is $O$-orthogonal to $V_n$. If $V_n \not \subseteq V$ there are elements in $W$ not orthogonal to $V_n$.\footnote{The (crucial) point that $P_n K$  and $K_n$ need not coincide on $W$ was overlooked in \cite{Chr09AWM}.}

The non-zero discrete eigenvalues of $a$ are the inverses of the non-zero eigenvalues
of $K_n$.

A minimal assumption for reasonable convergence properties is:
\begin{description}
\item[AS ] We say that $(X_n)$ are Approximating Subspaces of $X$ if:
\begin{equation}
\forall u \in X \quad \lim_{n \to \infty} \inf_{u' \in X_n} \| u -
u'\| = 0.
\end{equation}
\end{description}

\begin{proposition}\label{prop:asv}
If AS holds we also have:
\begin{align}
\forall u \in O & \quad \lim_{n \to \infty} \inf_{u' \in X_n} | u -
u'| = 0,\\
\forall v \in \overline V &\quad \lim_{n \to \infty} \inf_{v' \in V_n} | v -
v'| = 0,\\
\forall v \in V &\quad \lim_{n \to \infty} \inf_{v' \in V_n} \| v -
v'\| = 0.
\end{align}
\end{proposition}
\begin{proof}
(i) The first property holds by density of $X$ in $O$.

(ii) If we denote by $Q_n$ the $O$-orthogonal projection onto $W_n$ and $v\in \overline V$ is approximated in the $O$ norm by a sequence $u_n \in X_n$ we have:
\begin{equation}
 u_n -Q_n u_n \in V_n,
\end{equation}
and, since $Q_n v$ is orthogonal to both $v$ and $V_n$:
\begin{equation}
| v -  (u_n -Q_n u_n) | \leq  | (v - u_n) - Q_n (v-u_n)| \leq  |v-u_n| \to 0.
\end{equation}

(iii) Define $Q_n$ as above. In $X$ it is a projection with norm $1$. Pick $v \in V$ and choose a sequence $u_n \in X_n$ converging to $v$ in $X$. Remark that $Q_n v$ is orthogonal to both $v$ and $V_n$, also in $X$. We get:
\begin{equation}\label{eq:quasioptxv}
\| v -  (u_n -Q_n u_n) \| \leq  \| (v - u_n) - Q_n (v-u_n)\| \leq  2 \|v-u_n\| \to 0.
\end{equation}
This concludes the proof.
\end{proof}
Statements of the form ``there exists $C>0$ such that for all $n$, certain quantities $f_n$ and $g_n$ satisfy $f_n \leq C g_n$'', will be abbreviated $f_n \cleq g_n$. Thus $C$ might depend on the choice of sequence $(X_n)$ but not on $n$.

\section{Friedrichs, gaps and eigenvalue convergence\label{sec:frg}}
In this section, and for the rest of this paper, we assume that the spaces $(X_n)$ have been chosen so that AS holds.

Recall that $V_n$ is in general not a subspace of $V$. Below we will relate convergence of variational discretizations, to how well elements of $V_n$ can be approximated by those of $V$. More generally one is interested in what properties of elements of $V$, the elements of $V_n$ share. 

For non-zero subspaces $U$ and $U'$ of $X$ one defines the gap $\delta(U,U')$ from $U$ to $U'$ by:
\begin{equation}
\delta(U,U')= \sup_{u\in U}\inf_{ u'\in U'}\| u - u'\|/\|u\|.
\end{equation}
In what follows gaps will be computed with respect to the norm of $X$, unless otherwise specified.

In the above setting we define:
\begin{definition}\begin{description}
\item[ADK] We say that we have Approximating Discrete Kernels if:
\begin{equation}\label{eq:approxkernel}
\forall w \in W \quad \lim_{n \to \infty} \inf_{w' \in W_n} | w -
w'| = 0.
\end{equation}
\item[DF] Discrete Friedrichs holds if there is $C >0$
  such that:
\begin{equation}
\forall n \ \forall v \in V_n \quad |v|^2 \leq C a(v,v).
\end{equation}
\item[DC] Discrete Compactness holds if, from any subsequence of elements $v_n \in V_n$ which is bounded in $X$, one can extract a subsequence converging strongly in $O$, to an element of $V$.
\item[VG] We say that Vanishing Gap holds if:
\begin{equation}
\delta(V_n,V) \to 0 \textrm{ when } n \to \infty.
\end{equation}

\item[ODF] We say that the Optimal Discrete Friedrichs inequality holds if for some $C>0$ we have:
\begin{equation}
 \forall n \ \forall v \in V_n \quad |v| \leq C | \myP v|.
\end{equation} 
\end{description}
\end{definition}

This section discusses first the relations between these conditions, and second their relation to the convergence of $K_n$ to $K$ in various senses. The first three conditions are studied in particular in \cite{Kik89,CaoFerRaf00}, see also \cite{Bof10} \S 19. There are many possible variants of DC and we discuss several below. For an often overlooked subtlety in the definition of DC, see Remark \ref{rem:dcsubs}. VG appeared in \cite{Chr04MC,BufChr03,Buf05}. ODF is a condition we introduce here, as a counterpart to boundedness in $O$ of certain operators (mimicking so called commuting projections and Fortin operators) that will be discussed later, in \S \ref{sec:projections}.  Notice that ODF implies DF, due to (\ref{eq:frie}) and (\ref{eq:app}). A summary of the proved implications is provided at the end of this section, in diagram (\ref{eq:summary}).

Even ADK, which we will see is the weakest of the above hypotheses, is quite restrictive, but arises naturally in discretizations of complexes of Hilbert spaces by subcomplexes as will be explained in \S \ref{sec:hodge}. For the examples of Remark \ref{rem:exoneandtwo} the most important finite element spaces are those of \cite{Ned80}. We will see that they have all the above properties, of which ODF is the strongest. 

Strictness of proved implications will be discussed in \S \ref{sec:strict} below.

\subsection{Internal relations} We first study the implications between these conditions.

\begin{lemma}\label{lem:noadk}
Suppose ADK does not hold. Then there is a subsequence of elements $v_n \in V_n$  such that:
\begin{itemize}
\item $v_n$ converges weakly in $X$ to a non-zero element of $W$.
\item $Pv_n$ converges strongly to $0$ in X.
\end{itemize}
\end{lemma}
\begin{proof}
For any closed subspace $Y$ of $X$, let $Q[Y]$ be the $X$-orthogonal projection onto $Y$ defined by (\ref{eq:scalpx}). Suppose ADK is not satisfied. Choose $w \in W$,  $\epsilon >0$ and an infinite subset $N$ of $\bbN$ such that for all $n \in N$:
\begin{equation}\label{eq:wbb}
| w - Q[W_n] w| \geq \epsilon.
\end{equation}
By AS we have:
\begin{equation}\label{eq:wcn}
 Q[V_n] w + Q[W_n]w = Q[X_n] w  \to w \rmins{in} X. 
\end{equation}
Since $Q[V_n]w$ is bounded in $X$, we may suppose in addition that it converges weakly in $X$ to some $w'$. Since $(w - Q[W_n]w)$ then also converges weakly to $w'$, we must have $w' \in W$. 

If $w'= 0$, $Q[W_n]w$ converges weakly to $w$ and we can write:
\begin{align}
\| w\|^2 \leq& \liminf \| Q[W_n] w \|^2,\\
 =&  \liminf (\| w\|^2 - \| w - Q[W_n] w \|^2),\\
 \leq& \| w\|^2 - \epsilon^2.
\end{align}
This is impossible. So  $w' \neq 0$. 


We have, by (\ref{eq:wcn}):
\begin{equation}
a(Q[V_n] w, Q[V_n] w) = a(Q[X_n] w, Q[X_n] w) \to a(w,w) = 0,
\end{equation}
Hence, using (\ref{eq:app}), $PQ[V_n] w$ converges strongly in $X$ to $0$.

The subsequence $v_n = Q[V_n] w$ satisfies the stated conditions.
\end{proof}

\begin{proposition}\label{prop:adknsc} The following are equivalent:
\begin{itemize}
\item ADK.
\item For any subsequence $v_n \in V_n$ that converges weakly in $O$ to some $v\in O$, we have $v \in \overline V$.
\item For any subsequence  $v_n \in V_n$ that converges weakly in $X$ to some $v\in X$, we have $v \in V$.
\end{itemize}
\end{proposition}
\begin{proof}
(i) Suppose ADK holds. Let $v_n \in V_n$ converge weakly in $O$ to $v$. Pick $w \in W$ and choose $w_n \in W_n$ converging, in the $O$-norm, to $w$. Then we have:
\begin{equation}
\langle v, w \rangle = \lim_{n \to \infty} \langle v_n, w_n \rangle = 0,
\end{equation}
so that $v\in \overline V$. 

(ii) The third statement follows from the second by (\ref{eq:intersect}).

(iii) The third statement implies the first by Lemma \ref{lem:noadk}.
\end{proof}

The following corresponds to Proposition 2.21 in \cite{CaoFerRaf00}.
\begin{proposition}\label{prop:dfadk}
DF implies ADK.
\end{proposition}
\begin{proof}
Suppose ADK does not hold. Let $v_n$ be a subsequence as in Lemma \ref{lem:noadk}. Recalling (\ref{eq:app}), we get that:
\begin{equation}
a(v_n,v_n) = a(Pv_n, Pv_n) \to 0,
\end{equation}
but $v_n$ cannot converge strongly to $0$ in $O$. Hence DF is not satisfied. 
\end{proof}

\begin{remark}\label{rem:dcsubs}
So far the finite element literature has been rather cavalier about taking \emph{sub}-sequences from the outset in the formulation of DC, probably because it is rarely made clear what the set of indexes is, in the first place. See Remark 19.3 in \cite{Bof10}. Without this precision one cannot hope to deduce eigenvalue convergence, since one cannot rule out that good Galerkin spaces are interspaced with bad ones.  
\end{remark}

In the formulation of DC some authors, including \cite{CaoFerRaf00} and \cite{Bof10}, prefer not to impose that the limit is in $V$. Compactness is also frequently used in the form that a weak convergence implies a strong one. For completeness we state the following two alternatives to DC:
\begin{description}
\item[DC']  We say DC' holds if, from any subsequence of elements $v_n \in V_n$ which is bounded in $X$, one can extract a subsequence converging strongly in $O$.
\item[DC''] We say DC'' holds if, for any sequence of elements $v_n \in V_n$ which converges weakly in $X$, the convergence is strong in $O$.
\end{description}
Notice that in the last variant, we use sequences, not subsequences -- though we could have done that as well. These conditions are related as follows.
\begin{proposition}\label{prop:dcvariants} The following are equivalent:
\begin{itemize}
\item DC.
\item DC' and ADK.
\item DC'' and ADK.
\end{itemize}
\end{proposition}
\begin{proof}
(i) Suppose  DC holds. DC' trivially holds, so we focus on ADK. Consider a subsequence  $v_n \in V_n$ that converges weakly in $X$ to some $v\in X$. Extract from it one that converges strongly in $O$ to an element in $V$. Since this limit must be $v$, $v \in V$. This proves ADK by Proposition \ref{prop:adknsc}.

(ii) Let a sequence $v_n\in V_n$ converge weakly in $X$ to say $v$. If it does not converge strongly in $O$ to $v$ choose an $\epsilon >0$ and a subsequence for which $|v_n - v| \geq \epsilon$. No subsequence of this subsequence can converge strongly in $O$, since the limit would have to be $v$. Hence DC' implies DC''.

(iii) Suppose DC'' and ADK hold. Consider a subsequence $v_n \in V_n$ bounded in $X$. Extract from it one that converges weakly in $X$ to say $v$. By ADK, $v \in V$. For the indexes $n$ not pertaining to this subsequence, insert the best approximation of $v$ in $V_n$. Using Proposition \ref{prop:asv} we now have a sequence converging weakly in $X$ to $v$, so by DC'' it converges strongly in $O$. 

This proves DC.
\end{proof}
\begin{remark} Neither DC' nor DC'' implies ADK. Consider the case where $W$ is finite dimensional but non zero. Since both $W$ and $V$ are compactly injected into $O$, the injection $X \to O$ is compact, so that DC' and DC'' are automatically satisfied.  But one can construct subspaces $X_n$ satisfying AS, whose intersection with $W$ is zero, so that ADK is not satisfied.
\end{remark}

Corresponding to the well known fact that DC implies DF, we state:
\begin{proposition}\label{prop:vgdf}
VG implies DF.
\end{proposition}
\begin{proof}
 We have, for $v \in V_n$:
\begin{equation}
 \| v  - Pv \| \leq \delta (V_n, V) \| v \|,
\end{equation}
which gives:
\begin{equation}
(1 -\delta (V_n, V))  \| v \| \leq \| P v \|.
\end{equation}
For, say, $ \delta (V_n, V) \leq 1/2$ we then get:
\begin{equation}
 |v| \leq \| v \| \leq 2 \|Pv\| \cleq a(Pv,Pv)^{1/2} = a(v,v)^{1/2}, 
\end{equation}
which completes the proof.
\end{proof}

The following is an abstract variant of Proposition 7.13 in \cite{BufPer06}.
\begin{proposition}\label{prop:dcvg}
DC is equivalent to VG.
\end{proposition}
\begin{proof}
(i) Suppose VG holds. Suppose a subsequence of elements $v_n \in V_n$ is bounded in $X$. Then $Pv_n \in V$ is bounded in $X$, so it has a subsequence converging strongly in $O$ and weakly in $X$. By (\ref{eq:intersect}) the limit, call it $v$, is in $V$.  We have:
\begin{equation}
|v_n - Pv_n| \leq \delta(V_n,V)\|v_n\| \to 0,
\end{equation}
so that $v_n$ converges in $O$ to $v$. Hence DC holds.

(ii) Suppose that DC holds but not VG. We get a contradiction as follows. 
Choose $\epsilon >0$ and a subsequence of elements $v_n \in V_n$ which is bounded in $X$ but such that, for all $n$ pertaining to the subsequence:
\begin{equation}
\|v_n - P v_n\| \geq \epsilon.
\end{equation}
Extract a subsequence that converges strongly in $O$ to some $v \in V$. Then $(Pv_n)$ also converges to $v$ in $O$. We can now write:
\begin{equation}
|v_n - v| \geq |v_n - Pv_n| - |Pv_n - v| = \|v_n - P v_n\| - |Pv_n - v|,
\end{equation}
which is eventually bounded below by $\epsilon/2$, contradicting the convergence of $(v_n)$.
\end{proof}

As already indicated, ODF implies DF, but actually something stronger is true:
\begin{proposition}\label{prop:odfdc}
ODF implies DC.
\end{proposition}
\begin{proof}
Suppose ODF holds. Consider a subsequence $v_n\in V_n$, bounded in $X$. Then $Pv_n$ is bounded in $V$, so we may find a subsequence converging strongly in $O$ to some $v \in V$. Choose $u_n\in V_n$ converging strongly to $v$ in $O$ (Proposition \ref{prop:asv}). Then $Pu_n$ converges to $v$ in $O$. Using ODF we deduce:
\begin{align}
|v_n - u_n| &\cleq |Pv_n - Pu_n|,\\
& \cleq |v- P v_n| + |v - Pu_n| \to 0.
\end{align}
Therefore $v_n$ converges to $v$ strongly in $O$.
\end{proof}

One can also consider gaps defined by the norm of $O$ rather than $X$. Explicitly, for subspaces $U$ and $U'$ of $O$ we define:
\begin{equation}
\delta_O(U,U')= \sup_{u\in U}\inf_{ u'\in U'}| u - u'|/|u|.
\end{equation}

\begin{proposition}\label{prop:odfgap}
ODF is equivalent to the condition: there exists $\delta <1$ such that for all $n$, $\delta_O(V_n, \overline V) \leq \delta$.
\end{proposition}
\begin{proof}
Remark first that for all $n$ and all non-zero $v \in V_n$, $|v - Pv| < |v|$, so by compactness of the unit sphere of $V_n$, $\delta_O(V_n, \overline V) <1$.

(i) If the stated condition did not hold we would have a subsequence of elements $v_n \in V_n$ such that $|v_n|= 1$ and $|v_n - Pv_n|$ converges to $1$. Then we have:
\begin{equation}
 |P v_n|/|v_n| = (1 - |v_n - Pv_n|^2)^{1/2} \to 0.
\end{equation}
Hence ODF does not hold.

(ii) If on the other hand the stated condition holds we have, for all $n$ and $v \in V_n$:
\begin{equation}
|Pv|^2 = |v|^2 - |v - Pv|^2 \geq (1 - \delta^2) |v|^2,
\end{equation} 
and this gives ODF.
\end{proof}

\begin{remark} \label{rem:divcurl} In the context of $\curl$ problems with natural boundary conditions in convex domains (see Remark \ref{rem:exoneandtwo}), the ODF has the following equivalent formulation in terms of a negative norm: 
\begin{equation}
 \forall n \ \forall v \in V_n  \quad \|v\|_{\rmL^2(S)} \leq C \| \curl v \|_{\rmH^{-1}(S)}.
\end{equation} 
Here, $V_n$ is the subspace of vectorfields in $X_n$ that are orthogonal to discrete gradients (of functions that can be non-zero on the boundary). In this form the condition was discussed in connection with a discrete div curl lemma in \cite{Chr05}. An advantage of the present formulation is to avoid negative norms and boundary conditions.

In \cite{Chr09Calc} it was pointed out that for the discrete div curl lemma to hold, ODF is actually necessary, at least on periodic domains.
\end{remark}

\begin{remark}
An interpretation of results in \cite{Chr09Calc} (\S 3), is that for some standard spaces satisfying ODF (the edge elements of \cite{Ned80} for instance) we do \emph{not} not have that $\delta_O(V_n, \overline V)$ converges to $0$. 
\end{remark}

\subsection{Convergence of operators \label{sec:convop}}
We now explore the relation of the conditions DF and VG to the convergence of the operators $K_n$.
 
We first state, concerning $P_n$:
\begin{proposition}\label{prop:dfunifb} The following are equivalent:
\begin{itemize}
\item DF.
\item For any $u\in X$, $P_nu$ converges to $P u$ in $X$.
\item The $P_n$ are uniformly bounded $X \to X$.
\item The $P_n$ are uniformly bounded $V \to O$.
\end{itemize}
\end{proposition}
\begin{proof}
(i) Suppose DF holds, and pick $u \in X$. Choose $v_n \in V_n$ converging to $Pu$ in $X $ (Proposition \ref{prop:asv}).
We have:
\begin{align}
\| P_n u - v_n \|^2 &\cleq a(P_n u - v_n, P_n u - v_n) = a(Pu - v_n, P_n u - v_n)\\
 &\cleq \|Pu - v_n\| \|P_n u - v_n\|,
\end{align}
which gives (quasi-optimal) convergence of $P_n u$ to $Pu$.

(ii) Pointwise convergence of $P_n$ implies uniform boundedness $X \to X$ by the principle. Uniform boundedness $X \to X$ implies uniform boundedness $V \to O$, of course.

(iii) Suppose DF does not hold. Pick a  subsequence $v_n \in V_n$ such that $|v_n|=1$ but $a(v_n, v_n) \to 0$. Since $a(P v_n, P v_n) = a(v_n,v_n)$, the sequence $(P v_n)$ converges to $0$ in $V$. On the other hand $P_n P v_n = v_n$, whose $O$-norm is $1$. 

This precludes uniform boundedness of $ P_n : V \to O$.
\end{proof}

We state similar results for $K_n$, the only subtlety being the behaviour of $K_n$ on $W$. This particular point will be further expanded upon in Remark \ref{rem:knw}.
\begin{proposition} \label{prop:dfunifbk} The following are equivalent:
\begin{itemize}
\item DF.
\item For any $u \in X'$, $K_n u$ converges to $Ku$ in $X$.
\item The $K_n$ are uniformly bounded $X' \to X$.
\item The $K_n$ are uniformly bounded $O \to O$.
\end{itemize}
\end{proposition}
\begin{proof}
(i) Suppose DF holds. 

Pick $u \in V'$. Then $K_n u = P_n K u$ and $Ku \in X$, so by the preceding proposition $K_n u \to Ku$ in $X$.

Pick $u \in W$. By ADK (Proposition \ref{prop:dfadk}) we can choose a sequence $u_n \in W_n$ converging to $u$ in $W_n$. We have:
\begin{align}
\| K_n u \|^2 & \cleq a(K_nu, K_n u) = \langle u - u_n, K_n u \rangle\\
& \cleq |u - u_n| \, \|K_n u\|.
\end{align}
Hence $K_n u \to 0 = Ku$ in $X$.

This proves the second statement.

(ii) The second statement implies the third, which in turn implies the last.
 
(iv) When the last condition holds, the largest eigenvalue of $K_n$ is uniformly bounded, which implies DF.
\end{proof}

The above propositions show, as is well known, that DF is the key estimate for the convergence of the Galerkin method (\ref{eq:defkn}) for source problems. Turning to eigenvalue approximation, DF simply says that the smallest non-zero discrete eigenvalue of $a$ is bounded below, above zero. It is well known, since \cite{BofBreGas00}, that DF by itself is not enough to guarantee eigenvalue convergence. This has been the main reason for considering extra conditions, like DC \cite{Kik89}. 

To study the convergence of the spectral attributes (such as eigenvalues) of $K_n$, to those of $K$, a number of conditions have been used \cite{DunSch63,Kat80,Cha83}, in quite general settings that include unbounded operators.  The most convenient sufficient condition for us, appears to be: 
\begin{description}
\item[NC] We say Norm Convergence holds if $\|K - K_n\|_{O \to O} \to 0$.
\end{description}

In \cite{Osb75}, see \cite{BabOsb91} \S 7, it is shown directly that if NC holds, then for any non-zero eigenvalue $\lambda$ of $K$, any small enough disk around $\lambda$ in $\bbC$, the eigenvalues of $K_n$ inside it converge to $\lambda$, with corresponding convergence of eigenspaces in the sense of gaps. A key intermediate step is the norm convergence of the spectral projections associated with the disc, written as contour integrals. The argument easily extends to any contour in $\bbC$ that does not meet the spectrum of $K$ or enclose $0$.

In \cite{BofBreGas00} \S 5 a reciprocal is proved, for positive and injective $K$: a convergence of eigenvalues and spaces in the sense of gaps, implies NC.

Arguably DF is not necessary for eigenvalue convergence of the discretizations (\ref{eq:disceig}). Indeed one can imagine a scenario where some discrete nonzero eigenvalues of $a$ cluster around $0$, arbitrarily close as $n \to \infty$, with corresponding eigenspaces close to $W$, in the sense of gaps. This scenario need not be catastrophic in practice, for instance to someone who needs to identify resonances of an electromagnetic device. In the case where $W$ is finite dimensional this is even more patent: any approximating subspaces $X_n$ will yield a reasonable notion of eigenvalue convergence, even those that do not contain $W$. For such discretizations, $K_n$ is not even uniformly bounded in $O$. 

However our present theory is motivated by the situation where DF, and hence ADK, holds. \emph{We then take it for granted that eigenvalue and eigenspace convergence for (\ref{eq:disceig}), properly defined in terms of gaps, is equivalent to NC.}


In \cite{CaoFerRaf00} Definition 6.1, a notion of ``spurious free'' approximation is introduced. It is quite easily seen to imply DF. In Theorem 6.10 of that paper, it is shown that spurious free approximation is equivalent to DC (and AS). They use results of \cite{DesNasRap78I} which concern a weaker variant of NC, applicable when $K$ is bounded but not necessarily compact. 

We now prove the analogue result for us, namely that NC is equivalent to VG.

\begin{lemma}\label{lem:gapb}
We have:
\begin{equation}\label{eq:deltab}
\delta(V_n,V) \leq \|K_n\|_{W \to X}. 
\end{equation}
\end{lemma}
\begin{proof}
For any $v \in V_n $ we have:
\begin{align}
|v -Pv|^2 &= \langle v - Pv, v \rangle,\\
&= a(K_n(v-Pv), v),\\
&\leq \|K_n(v - Pv)\| \, \|v\|,\\
&\leq \|K_n \|_{W \to X} \,|v-Pv | \, \|v\|.
\end{align}
We deduce:
\begin{equation}
\|v-Pv\| \leq \|K_n \|_{W \to X}  \| v\|.
\end{equation}
This proves the claim.
\end{proof}

\begin{proposition}\label{prop:equiv}
If DF holds then:
\begin{equation}\label{eq:convv}
\|K -K_n\|_{\overline V \to X} \to 0,
\end{equation}
and:
\begin{equation}\label{eq:deltaequiv}
\|K_n\|_{W \to X} \simeq \delta(V_n,V),
\end{equation}
\end{proposition}
\begin{proof}
By DF, the $P_n$ are projectors $X \to V_n$ which converge pointwise to $P$ (Proposition \ref{prop:dfunifb}). Since $K : O \to X$ is compact, Lemma \ref{lem:pointnorm} in the appendix then gives $\|K - P_nK\|_{O \to X} \to 0$, which gives (\ref{eq:convv}) since $K_n$ and $P_nK$ coincide on $\overline V$, see (\ref{eq:knpnk}).

Pick $u \in W$. We have:
\begin{align}
a(K_nu,K_nu) & = \langle u , K_n u \rangle,\\
& = \langle u, (I-P)K_n u \rangle,\\
& \leq |u| \, \delta (V_n,V) \, \| K_n u\|.
\end{align}
Since we also have, by DF:
\begin{equation}
\| K_n u\|^2 \cleq a(K_nu,K_nu),
\end{equation}
we get:
\begin{equation}
\| K_n u \| \cleq \delta (V_n,V)\, |u|.
\end{equation}
Combined with (\ref{eq:deltab}) this gives (\ref{eq:deltaequiv}).

\end{proof}

\begin{proposition}\label{prop:oconvvg}
The following are equivalent:
\begin{itemize}
\item VG.
\item $\|K - K_n\|_{O \to X}$ tends to $0$.
\item NC.
\end{itemize}

\end{proposition}
\begin{proof}
(i) If VG holds then DF holds. Applying Proposition \ref{prop:equiv} shows that $\|K - K_n\|_{O \to X}$ tends to $0$.

(ii) We have:
\begin{equation}
\|K - K_n\|_{O \to O} \leq \|K - K_n\|_{O \to X}.
\end{equation}

(iii) Suppose that $\|K - K_n\|_{O \to O} \to 0$.\\
For $u \in W$ we have:
\begin{equation}  
a(K_n u, K_nu) = \langle u, K_nu \rangle  \leq |u|\, |K_n u|.
\end{equation}
Therefore we have:
\begin{equation}
\|K_n \|_{W \to X} \to 0.
\end{equation}
By Lemma \ref{lem:gapb}, VG holds.
\end{proof}

\begin{remark}\label{rem:knw}
When DF holds, Proposition \ref{prop:equiv} shows that the question of eigenvalue convergence can be stated entirely in terms of the behaviour of $K_n$ on $W$, as a distinction between pointwise and uniform convergence of $K_n$. For any $u$ in $W$, $K_n u \to 0$ in $O$ (Proposition \ref{prop:dfunifbk}), but eigenvalue convergence holds iff $\|K_n\|_{ W \to O}$ tends to $0$. 
\end{remark}

\subsection{Summary} The following diagram represents the main implications proved in this section. The numbers refer to propositions.
\begin{equation}\label{eq:summary}
\xymatrix{
ODF \ar@{=>}[d]^{\ref{prop:odfdc}}\ar@{<=>}[r]^{\ref{prop:odfgap}} &SG  \\
DC \ar@{<=>}[r]^{\ref{prop:dcvg}} & VG \ar@{<=>}[r]^{\ref{prop:oconvvg}}  \ar@{=>}[d]^{\ref{prop:vgdf}} & NC\\
& DF \ar@{=>}[d]^{\ref{prop:dfadk}}\ar@{<=>}[r]^{\ref{prop:dfunifbk}} &PC   \\
&ADK
}
\end{equation}
PC stands for Pointwise Convergence of the solution operators $K_n$, as in the second point of Proposition \ref{prop:dfunifbk}. SG stands for the  Small Gap condition for $V_n$ in $O$, considered in Proposition \ref{prop:odfgap}. 

\section{Compatible operators and convergence criteria\label{sec:projections}}
 
In the finite element setting the most useful convergence results for variational discretizations are derived from projection operators such as those associated with the degrees of freedom. A key part of our investigation is therefore to relate the convergence criteria discussed above to the existence of such operators. For this purpose we have found it useful to introduce the following notion of a compatible operator.

\begin{definition}
We say that an operator $Q_n :X \to X_n$ is \emph{compatible} with $X_n$, if:
\begin{itemize}
\item it maps $W$ into $W_n$,
\item for any $u \in X_n$, $u - Q_n u \in W_n$.
\end{itemize}
 A sequence of operators $Q_n$ is compatible if for each $n$, $Q_n$ is compatible with $X_n$. 
\end{definition}
For instance the $P_n$ defined by (\ref{eq:defpn}) are compatible operators. They have range $V_n$ and map $W$ to $0$. Also, projections onto $X_n$ satisfying commuting diagrams, as will be detailed in \S \ref{sec:hodge} below, are compatible. They have range $X_n$ and project $W$ onto $W_n$. 

For projections $Q_n$ onto $X_n$, the second condition in the above definition of compatibility is trivially satisfied.

\begin{remark}\label{rem:fortin}
In the discussion of \cite{Bof10}, so-called Fortin operators play a central role (see \S 14). They are operators $Q_n: X \to X_n$ such that:
\begin{equation}
\forall u \in X \ \forall v \in V_n \quad a(u - Q_n u, v) = 0,
\end{equation}
which then holds for all $v \in X_n$. Thus $P_n$ is the only Fortin operator with range $V_n$. Any Fortin operator is compatible in the above sense but the reciprocal is false~: In the context of $\curl$ problems a Fortin operator would need to map the $\curl$ according to the $\rmL^2$ projection. In general, commuting projections don't have this property.
\end{remark}

To deduce convergence results from compatible operators, norm properties are necessary. We state:
\begin{definition}
\begin{description}
\item[FCO] Friedrichs Compatible Operators are compatible sequences of operators which are uniformly bounded $X \to X$.
\item[ANCO] Aubin-Nitsche Compatible Operators are compatible operators $Q_n$ such that $\|\id - Q_n\|_{V \to O}$ converges to $0$. In other words ANCO states that for some $\epsilon_n \to 0$ we have:
\begin{equation}\label{eq:anest}
\forall v \in V \ \forall n \quad |v - Q_nv | \leq \epsilon_n \| v\|.
\end{equation}
\item[TCO] Tame Compatible Operators are compatible operators which are uniformly bounded $O \to O$. We say that TCO holds whenever such operators exist.
\end{description}
\end{definition}
We have already encountered uniform boundedness $X \to X$, as in FCO, in Proposition \ref{prop:dfunifb}, where it was related to DF. The kind of norm convergence stipulated by ANCO corresponds to the so-called Fortid property of \cite{BofBreGas97}, see \cite{Bof10} \S 14 (the Fortin operator should converge to the identity in a certain norm). The construction of commuting projections that are tame in the sense of being uniformly bounded $O \to O$, is a more recent tool for the analysis of mixed finite elements, see \S \ref{sec:hodge} below for references. They appear to be a convenient, but not necessary, tool for proving eigenvalue convergence.

The following result can be compared with \S 3.3 in \cite{ArnFalWin10}, which treats commuting projections for Hilbert complexes.
\begin{proposition}\label{prop:dfksp}
The following are equivalent:
\begin{itemize}
\item DF.
\item The projectors $P_n$ are FCO.
\item There exist FCO. 
\end{itemize}
\end{proposition}
\begin{proof}
In view of Proposition \ref{prop:dfunifb} we only need to show that the last condition implies the first.
Suppose that $Q_n$ are FCO. Then for any $u \in V_n$ we have:
\begin{equation}
u - Q_n \myP u = (u - Q_n u) + Q_n(u-\myP u)\in W_n,
\end{equation}
so that $u - Q_n \myP u$ and  $u$ are orthogonal. We deduce:
\begin{equation}
\| u \|  \leq \| Q_n \myP u \| \cleq \| \myP u\|,
\end{equation}
hence:
\begin{equation}
|u|^2 \leq \|u\|^2 \cleq a(Pu,Pu) = a(u,u),
\end{equation}
which is  DF.
\end{proof}

The following result can be compared with \S 4 in \cite{BofBreGas97} which treats eigenvalue problems in saddlepoint form. That setting is more general than ours, but the analysis relies on additional notions not explicitly introduced here (weak and strong approximability).
\begin{proposition} \label{prop:vgancp} The following are equivalent:
\begin{itemize}
\item VG. 
\item The projectors $P_n$ are ANCO.
\item There exist ANCO.
\end{itemize}
\end{proposition}

\begin{proof}
(i) Suppose VG holds. We shall show that $P_n$ are ANCO. For $v \in V$ we write:
\begin{equation}\label{eq:twoterms}
|v - P_nv| \leq |P_nv - P P_nv| + |v - P P_nv|.
\end{equation}
For the first term on the right hand side we remark that:
\begin{equation}
|P_nv - P P_nv| \leq \delta (V_n,V) \|P_n v\| \cleq \delta(V_n,V) \|v\|.
\end{equation}
For the second term we use that for all $v' \in \overline V$:
\begin{align}
\langle v - P P_nv, v' \rangle & = a(v-P P_n v, Kv'),\\
&= a(v - P_n v, K v' - P_n Kv'),\\
&\leq \| v - P_n v\| \, \| Kv' - P_n Kv'\|,\\
& \cleq \|v\| \, \|K-P_n K\|_{O \to X} |v'|.
\end{align} 
This gives:
\begin{equation}
| v - P P_nv| \cleq \|K-P_n K\|_{O \to X} \|v\|.
\end{equation}
From the compactness of $K: O \to X$ and Lemma \ref{lem:pointnorm} in the appendix, one concludes that  $\|K-P_n K\|_{O \to X}$ converges to $0$.

Combining the bounds for the two terms in the right hand side of (\ref{eq:twoterms}), we obtain that $P_n$ are ANCO.

(ii) Suppose $Q_n$ are ANCO, so that we have a sequence $\epsilon_n \to 0$ verifying:
\begin{equation}
\forall v \in V \ \forall n \quad |v - Q_n v| \leq \epsilon_n \| v \|.
\end{equation}
Now for $v \in V_n$ we have:
\begin{equation}
v - Q_nPv = (v - Q_n v) +  Q_n(v -Pv) \in W_n,
\end{equation}
and, since both $v$ and $Pv$ are orthogonal to $W_n$:
\begin{equation}\label{eq:dischodgest}
|v - Pv | \leq |P v - Q_n Pv| \leq \epsilon_n \|Pv\|.
\end{equation}
This gives VG.
\end{proof}

\begin{remark}
Proposition \ref{prop:equiv} shows that when DF holds we have:
\begin{equation}
\|K - K_n\|_{O \to X} \cleq \max \{ \delta (V_n,V), \|K - P_nK\|_{O \to X} \}.
\end{equation}
Part (ii) of the proof of Proposition \ref{prop:vgancp} gives an estimate for $\delta(V_n,V)$ from any ANCO. On the other hand $P_n$ converges pointwise quasi-optimally to $P$ in $X$ (see Proposition \ref{prop:dfunifb} and equation (\ref{eq:quasioptxv})), so that the second term is bounded by the approximation rate:
\begin{equation}
\sup_{u \in O} \inf_{u_n \in X_n} \|K u - u_n\|/ |u|.
\end{equation}
Combining these remarks gives convergence \emph{rates} for $\| K - K_n\|_{O \to X}$.
\end{remark}

\begin{remark} To establish VG or some related results one can get away with even less. For instance in the finite element literature one has used operators $Q_n$ that are projections onto $X_n$ mapping (a dense subset of) $W$ to $W_n$, but that are not defined on all of $X$. The crucial point is that $Q_n$ should be defined on  $P X_n$ and that an estimate of the type (\ref{eq:dischodgest}) still holds. This is done in particular in Lemma 4.1 of \cite{CiaZou99}, which involves a local finite dimensionality argument.
\end{remark}

\begin{proposition}\label{prop:odftcp} The following are equivalent:
\begin{itemize}
\item ODF.
\item There exist TCO.
\end{itemize}
\end{proposition}
\begin{proof}
(i) Suppose ODF holds. In $V_n \oplus W$ let $Q_n$ be the projection with range $V_n$ and kernel $W$. For $u \in V_n \oplus W$ we have:
\begin{equation}
\myP (u-Q_n u)= 0.
\end{equation}
Using ODF we deduce:
\begin{equation}
|Q_n u| \cleq |\myP Q_n u| = |\myP u| \leq |u|.
\end{equation}
For any closed subspace $Z$ of $O$, let $P[Z]$ denote the $O$-orthogonal projection onto it. Define:
\begin{equation}\label{eq:rdef}
R_n = Q_n P[V_n \oplus W].
\end{equation}
They have range in $V_n$, map $W$ to $0$ and are uniformly bounded in $O$, so constitute TCOs.

(ii) We reason as in the proof of Proposition \ref{prop:dfksp}.  Suppose that $Q_n$ are TCOs. Then for any $u \in V_n$ we have:
\begin{equation}
u - Q_n \myP u = (u - Q_n u) + Q_n(u-\myP u)\in W_n.
\end{equation}
Hence $u - Q_n \myP u$ is orthogonal to $u$ in $O$ and we have:
\begin{equation}
| u | \leq | Q_n \myP u | \cleq | \myP u |.
\end{equation}
This gives ODF.
\end{proof}

\begin{remark}
Notice that we have proved that if TCO holds, then ODF holds by Proposition \ref{prop:odftcp}, so DC holds by Proposition \ref{prop:odfdc}, so VG holds by Proposition \ref{prop:dcvg}, so the $P_n$ are ANCO by Proposition \ref{prop:vgancp} -- however they will in general not be tame.  For instance, the variational $\rmH^1$ projections onto continuous piecewise affine finite element functions are not continuously extendable to $\rmL^2$ and a fortiori not $\rmL^2$ stable. 
\end{remark}

\begin{remark}
From Lemma \ref{lem:pointnorm} it follows that projections that are tame also satisfy the Aubin-Nitsche property, but in general TCOs do not need to be FCOs. On the other hand, tame commuting projections, discussed below, are uniformly bounded in both $X$ and $O$. Thus the notion of TCO is weaker than that of tame commuting projections.
\end{remark}

\begin{remark} If ODF holds, one can modify the definition (\ref{eq:rdef}) to:
\begin{equation}
R_n = Q_n P[V_n \oplus W] + P[W_n].
\end{equation} 
Restricted to $V_n$ and $W_n$,  $R_n$ is the identity, so that now $R_n$ is a projection with range $X_n$ instead of $V_n$. It is also uniformly bounded $O \to O$, and maps $W \to W_n$. Still, boundedness $X \to X$ does not seem guaranteed and this construction does not appear to provide tame commuting projections in general.
\end{remark}

\section{Additional remarks\label{sec:add}}

\subsection{Hodge Laplacian and commuting diagrams\label{sec:hodge}}
To study the de Rham complex on a manifold $S$ of dimension $n$, let the space $O^k$ consist of differential $k$-forms in $\rmL^2(S)$, and $X^k$ consist of the elements of $O^k$ with exterior derivative in $O^{k+1}$. We have a complex of Hilbert spaces, linked by the exterior derivative:
\begin{equation}
\xymatrix{
X^0 \ar[r]^\rmd & X^1 \ar[r]^\rmd & \ldots \ar[r]^\rmd & X^n.
}
\end{equation}

More generally consider a  Hilbert complex of spaces $X^k$ ($k \in \bbZ$). There is a generic notion of Hodge Laplacian at each index $k$, which can be given a weak formulation involving $X^k$ and $X^{k-1}$, see \cite{ArnFalWin10}. The eigenvalue problem can be decoupled into three problems: one semidefinite eigenvalue problem on $X^k$, a similar one for $X^{k-1}$ and one concerning harmonic forms. The discretization of these are all analyzed in terms of commuting projections that are tame (uniformly bounded in $O^k$) in \S 3.6 of \cite{ArnFalWin10}. 

Tame commuting projections have been constructed in concrete examples \cite{Sch08,Chr07NM,ArnFalWin06,ChrWin08}, corresponding to mixed finite elements for the de Rham complex, in the $h$-version. It is still an open problem if they exist in the $p$- and $hp$-versions of the finite element method. However DC has been verified for important problems in that context \cite{BofCosDauDem06,BofEtAl11}. In \S \ref{sec:strict} we will check, at an abstract level, that VG does not imply ODF. Therefore, as a tool to study eigenvalue convergence, tame commuting projections are more powerful than what is strictly necessary, to the extent that they might not exist. But recall from Remark \ref{rem:divcurl} that ODF is necessary for some non-linear estimates. Remark \ref{rem:lqco} gives additional motivation for constructing commuting projections that are stable in weak norms.

We show here how each of the above mentioned semi-definite eigenvalue problems deduced from Hodge Laplacians fits into the framework adopted in this paper, indicating in particular how commuting projections lead to compatible operators.

Suppose that for $k= 0,1$ we have Hilbert spaces $(O^k)$ and $(X^k)$ which can be arranged as follows:
\begin{equation}
\xymatrix{
O^0   & O^1  \\
X^0 \ar[r]^\rmd \ar[u]& X^1\ar[u]}
\end{equation}
The vertical maps are inclusions and the horizontal map $\rmd: X^0 \to X^{1}$ is bounded.

The scalar product on $O^k$ is denoted $\langle \cdot , \cdot \rangle_k$ and we put, for $u,u' \in X^0$:
\begin{equation}
a(u,u') = \langle \rmd u , \rmd u' \rangle_{1}.
\end{equation}
We suppose that the triplet $(O^0, X^0, a)$ complies with the setting of \S \ref{sec:setting}. In particular $X^0$ is dense in $O^0$ and the norm on $X^0$ can be taken to be defined by:
\begin{equation}
\|u\|^2 = \langle u, u \rangle_0 + \langle \rmd u, \rmd u \rangle_1.
\end{equation}
We split as before:
\begin{equation}
X^0 = V^0 \oplus W^0.
\end{equation}
We remark that $W^0$, the kernel of $a$, is also the kernel of $\rmd : X^0 \to X^1$. It is part of the hypotheses that the injection $V^0 \to O$ must be compact.

We suppose that for $k= 0, 1 $ we have a sequence of subspaces $(X^k_n)$ of $X^k$, such that $\rmd: X^0_n \to X^{1}_n$. We split them as before $X^0_n = V^0_n \oplus W^0_n$. One sees that $W^0_n$ is the kernel of $\rmd$ on $X^0_n$. 

To check ADK we can do the following. Suppose we have dense subspaces $Y^k$ of $X^k$, which are also Banach spaces with continuous inclusions, and that we have projections $Q_n^k : Y^k \to X^k_n$ such that the following diagrams commute:
\begin{equation}\label{eq:comdiag}
\xymatrix{
Y^0 \ar[r]^\rmd \ar[d]^{Q^0_n}& Y^1  \ar[d]^{Q^1_n}\\
X^0_n \ar[r]^\rmd & X^1_n 
}
\end{equation}
Then $Q^0_n$ maps $W^0\cap Y^0$ to $W^0_n$. This guarantees ADK to hold, under the mild assumption that $W^0 \cap Y^0$ is dense in $W^0$ and $Q^0_n$ uniformly bounded $Y^0 \to O^0$. 

If, in diagram (\ref{eq:comdiag}), $Y^k = X^k$, then $Q^0_n$ is a compatible operator according to our definition. In particular if it is FCO, ANCO or TCO, corresponding estimates hold, guaranteeing convergence of various variational problems, as detailed in the preceding sections.

In the mixed finite element setting, this reasoning can be applied to the canonical interpolators. The Sobolev injection theorems provide dense Banach spaces $Y^k$ on which the degrees of freedom are well defined, proving ADK. Unfortunately canonical interpolators are not in general well defined on $X^k$, even less so on $O^k$. In the above cited references, concerning finite element exterior calculus, tame commuting projections have been constructed by composing canonical interpolators, with smoothing operators that commute with the exterior derivative.

\begin{remark}\label{rem:lqco} For the de Rham complex, tameness of commuting projections is boundedness with respect to the $\rmL^2$ norm. Commuting projections that are $\rmL^q$ bounded for other real $q$, are of interest for non-linear problems. For instance they were constructed and used in \cite{ChrMunOwr11} (\S 5.4) to prove discrete translation estimates and Sobolev injections, extending \cite{KarKar11} and \cite{BufOrt09}. 
\end{remark}

\subsection{Equivalent forms\label{sec:eqfor}}
For the case of electromagnetics, as in equation (\ref{eq:maxwell}) and Remark \ref{rem:exoneandtwo},  it is proved in \cite{CaoFerRaf00} (Proposition 2.27) that DC holds for electromagnetic coefficients $\epsilon$ and $\mu$ that are both scalar and equal to $1$,  if and only if DC holds for coefficients which are $3\times 3$ symmetric matrix fields which are bounded and uniformly positive definite.
In this section we give an abstract variant of this result.

Suppose that, in addition to the previous spaces and forms, we have a symmetric bilinear form $\langle \cdot, \cdot \rangle^\mys$ on $O$ and a symmetric bilinear form $a^\mys$ on $X$, which are equivalent to $\langle \cdot , \cdot \rangle$ and $a$:
\begin{align}
\forall u \in O& \quad  \langle u, u \rangle \cleq  \langle u  , u \rangle^\mys \cleq \langle u, u \rangle,\\ 
\forall u \in X& \quad a(u,u) \cleq a^\mys(u ,u) \cleq a(u,u).
\end{align}
These bilinear forms can be used to redefine the norm of $O$ as well as that of  $X$. Notice that $a$ and $a^\mys$ have the same kernel $W$. Applying the previous constructions we get new splittings, with obvious definitions:
\begin{align}
X  & = V^\mys \oplus W,\\
O  & = \overline V^\mys \oplus W.\label{eq:onewsplit}
\end{align}
We are interested in the eigenvalue problem these forms define, and its discretization on the same sequence of spaces $X_n$ as before. They get a new splitting:
\begin{equation}
X_n  = V_n^\mys \oplus W_n.
\end{equation}

First we check:
\begin{proposition}\label{prop:vsac}
The injection $V^\mys \to O$ is compact.
\end{proposition}
\begin{proof}
From (\ref{eq:onewsplit}) it follows that the projector $P$ (defined by $\langle \cdot, \cdot \rangle$ and $a$) induces a bijection $\overline V^\mys \to \overline V$ which is continuous in $O$. By the open mapping theorem its inverse, denoted $J$, is also continuous. 


Suppose $E$ is a bounded subset of $V^\mys$. Then $PE$ is bounded in $V$, hence relatively compact in $\overline V$. Therefore $JPE$ is relatively compact in $\overline V^\mys$. But $E = JPE$ so that $E$ is relatively compact in $O$. 
\end{proof}

Now one can compare the meaning of DF, VG and ODF with respect to $\langle \cdot , \cdot \rangle$ and $a$, with their corresponding statements for $\langle \cdot , \cdot \rangle^\mys$ and $a^\mys$, which we abbreviate as DF\myss, VG\myss, ODF\myss.

 \begin{proposition} We have:
\begin{itemize}
\item DF and DF\myss are equivalent.
\item VG and VG\myss are equivalent.
\item ODF and ODF\myss are equivalent.
\end{itemize}
\end{proposition}
\begin{proof}
We use the notion of compatible operator, which is the same for the two choices of bilinear forms.

(i) The first statement follows from Proposition \ref{prop:dfksp} and the third from Proposition \ref{prop:odftcp}.

(ii) We now prove that VG implies VG\myss, which yields the second statement, since the two choices of bilinear forms play symmetric roles. Define $P$ and $P_n$ by $\langle \cdot , \cdot \rangle$ and $a$, as before. We also let $Q_n$ denote the projection onto $W_n$, defined by the bilinear form $\langle \cdot, \cdot \rangle$ and put:
\begin{equation}
R_n = P_n + Q_n.
\end{equation}
Notice that $R_n$ is compatible operator, in fact a projection onto $X_n$ mapping $W$ onto $W_n$.

For $v \in V^\mys$ we write:
\begin{align}
v - R_n v & = (I-R_n)P v+ (I-R_n)(I - P) v,\\
& = (I-P_n)Pv + (I-Q_n) (I-P)v.
\end{align}
The first term on the right hand side is handled by Proposition \ref{prop:vgancp}, which yields an estimate (with $\epsilon_n \to 0$):
\begin{equation}
| (I-P_n)Pv | \cleq \epsilon_n \| v \|.
\end{equation}
For the second term, we notice that by Proposition \ref{prop:vsac},  $I-P$ defines a compact operator $V^\mys \to O$, with values in $W$. Recall that VG implies ADK, so that $Q_n$ converges pointwise to the identity on $W$. We can therefore apply Lemma \ref{lem:pointnorm}, to get an estimate (with $\epsilon_n' \to 0$):
\begin{equation}
| (I-Q_n)(I-P)v | \cleq \epsilon'_n \| v \|.
\end{equation}

Together these estimates prove that $R_n$ is an ANCO with respect to $\langle \cdot , \cdot \rangle^\mys$ and $a^\mys$. Therefore  VG\myss holds.
\end{proof} 

\subsection{Strictness of implications\label{sec:strict}}

\begin{remark} If $W$ is finite dimensional, ADK implies that for $n$ big enough $W_n = W$ so $V_n \subseteq V$, so that ODF holds.
\end{remark}

\emph{In the following propositions we suppose that $W$ is infinite-dimensional.} We then prove that among the above proved implications, those that were not proved to be equivalences indeed are not.

\begin{proposition} Spaces satisfying AS and ADK can be perturbed along one sequence of directions so as to still satisfy AS and ADK but not DF.
\end{proposition} 
\begin{proof}
Let $X_n$ be spaces satisfying AS and ADK. Pick $v_n \in V_n$ with norm $1$ in $X$ but converging weakly to $0$ in $X$. Pick $w_n$ in $W$, orthogonal to $X_n$ and such that $|w_n|=1$. Choose a sequence $\epsilon_n>0$, converging to $0$ faster than $|v_n|$. That is, we impose $\epsilon_n = o(|v_n|)$. Define an operator $Z_n$ by:
\begin{equation}
Z_n u = u - \epsilon_n^{-1}\langle u , v_n \rangle/|v_n|\, w_n.
\end{equation}
It is injective on $X_n$. Put $\tilde X_n = Z_n X_n$. 
We claim that $\tilde X_n$ satisfies AS and ADK but not DF.

Remark that for $u,u' \in X_n$:
\begin{equation}
a(Z_n u, Z_n u') = a(u,u'),
\end{equation}
and that $\tilde X_n$ splits according to (\ref{eq:split}) as follows:
\begin{equation}
\tilde W_n = W_n \textrm{ and } \tilde V_n = Z_n V_n.
\end{equation}
Then the spaces $\tilde X_n$ satisfy AS and ADK. Remark that:
\begin{equation}
a(Z_n v_n, Z_n v_n) = a(v_n,v_n) \cleq 1 \quad \textrm{but} \quad |Z_nv_n|^2 = |v_n|^2 (1 + \epsilon_n^{-2}) \to \infty,
\end{equation}
so that DF does not hold.
\end{proof}

In \cite{BofBreGas00} finite element spaces satisfying DF but not DC are discussed. Here is an abstract variant.

\begin{proposition}\label{prop:dcnotvg} Spaces satisfying AS and DF can be augmented by one vector so as to still satisfy AS and DF but not DC.
\end{proposition}
\begin{proof}
Suppose spaces $X_n$ satisfy DF. Pick a sequence $v_n \in V$ such that $\|v_n\| = 1$ and $v_n$ is orthogonal to $V_n$ in $O$ and $X$. Pick also a sequence $w_n\in W$ such that $|w_n|=1$ and $w_n$ is orthogonal to $W_n$ in $O$. Put now $u_n = v_n + w_n$ and set:
\begin{equation}
\tilde X_n = X_n \oplus \bbR u_n.
\end{equation}
We claim that $\tilde X_n$ satisfies DF but not DC. Indeed $\tilde X_n$ splits as follows:
\begin{equation}
\tilde W_n  = W_n \textrm{ and } \tilde V_n = V_n\oplus \bbR u_n, 
\end{equation}
and DF holds. The sequence $(u_n)$ converges to $0$ weakly in $X$ but not strongly in $O$.
\end{proof}

The following shows that TCO, while sufficient, is not necessary for eigenvalue convergence.
\begin{proposition}
Spaces satisfying AS and VG can be augmented by one vector so as to still satisfy AS and VG but not ODF.
\end{proposition}
\begin{proof}
Let spaces $X_n$ satisfy VG. Choose $v_n$ in $V$ orthogonal to $X_n$ such that $\|v_n\|=1$. Then $v_n$ converges weakly to $0$ in $X$ and strongly in $O$. Choose also $w_n$ in $W$ orthogonal to $X_n$ and such that $|w_n|= |v_n|^{1/2}$.

Now put $u_n = v_n + w_n$ and set:
\begin{equation}
\tilde X_n = X_n \oplus \bbR u_n.
\end{equation}

As before, $\tilde X_n$ splits according to (\ref{eq:split}) as follows:
\begin{equation}
\tilde W_n  = W_n \myand \tilde V_n = V_n\oplus \bbR u_n, 
\end{equation}
We have:
\begin{equation}
|u_n - P u_n| = |w_n| \to 0 \myand \|u_n\| \geq \|v_n\| = 1, 
\end{equation}
hence VG still holds.

But we also have:
\begin{equation}
|Pu_n|/|u_n| = |v_n|/(|v_n|^2 + |w_n|^2)^{1/2} = 1/ (1 + 1/|v_n|)^{1/2} \to 0.
\end{equation}
so ODF does not hold.
\end{proof}

\section{Appendix}
A variant of the following Lemma was already used in \cite{Kol77}. We temporarily forget about
the previous notations.
\begin{lemma}\label{lem:pointnorm}
Let $X$, $Y$ and  $Z$ be three Hilbert spaces, $K: X \to Y$ compact,
$A: Y \to Z$ bounded and $A_n: Y \to Z$ a uniformly bounded sequence of
operators such that:
\begin{equation}
\forall u \in X \quad \|AKu - A_n K u\|_{Z} \to 0.
\end{equation}
Then:
\begin{equation}
\| AK - A_n K\|_{X \to Z} \to 0 .
\end{equation}
\end{lemma}

For instance (due to the uniform boundedness principle)  this lemma can be applied if $(A_n)$ is a sequence of bounded operators $Y \to Z$ converging pointwise (in norm) to some $A$.

\section{\label{sec:ack} Acknowledgements}
The first author thanks Annalisa Buffa and Daniele Boffi for helpful discussions.

This work, conducted as part of the award ``Numerical analysis and simulations of
geometric wave equations'' made under the European Heads of Research Councils and European Science
Foundation EURYI (European Young Investigator) Awards scheme, was supported by funds from the
Participating Organizations of EURYI and the EC Sixth Framework Program.

\bibliography{../Bibliography/alexandria}{}
\bibliographystyle{plain}

\end{document}